\documentclass[11pt]{article}
\usepackage[margin=1in]{geometry}
\usepackage[T1]{fontenc}
\usepackage{lmodern,amsmath,amssymb,amsthm}
\usepackage[hidelinks]{hyperref}
\usepackage{xurl}
\newcommand{\E}{\mathbb E}
\newcommand{\Var}{\operatorname{Var}}
\newtheorem{theorem}{Theorem}
\newtheorem{corollary}[theorem]{Corollary}
\newtheorem{lemma}[theorem]{Lemma}
\newtheorem{proposition}[theorem]{Proposition}
\title{Variance of random greedy independent sets in triangle-free graphs}
\author{Mubin Shaikh\\
{\normalsize Independent Researcher}\\
{\small \href{mailto:shaikhmubin572@gmail.com}{\nolinkurl{shaikhmubin572@gmail.com}}}\\
{\small ORCID: \href{https://orcid.org/0009-0005-1290-7114}{0009-0005-1290-7114}}}
\date{September 13, 2026}
\hypersetup{pdftitle={Variance of random greedy independent sets in triangle-free graphs},pdfauthor={Mubin Shaikh}}
\begin{document}
\maketitle
\begin{abstract}
Inspect the vertices of a finite simple graph in uniformly random order, accepting each vertex if none of its neighbors has previously been accepted. Let $X_G$ be the number of accepted vertices. For every triangle-free graph with $n\ge2$ vertices and $e$ edges, we prove
\[
 \Var(X_G)\le e\left(\frac{n-2}{n}\right)^2,
\]
with equality precisely for edgeless graphs and connected stars. In particular, among trees of a given order, the star uniquely maximizes the variance, with value $(n-1)(n-2)^2/n^2$. Known expected vertex-deletion stability already yields the elementary baseline $\Var(X_G)\le e$. We obtain the sharp finite-order refinement by combining a stronger centered first-choice estimate with a triangle-free edge-count identity in the law of total variance.
\end{abstract}

\section{Introduction}

Random-order greedy acceptance produces a maximal independent set, but generally not a maximum independent set. Its cardinality is random even when the underlying graph is fixed. Equivalently, assign independent uniform arrival times in $[0,1]$ to the vertices and occupy a vertex on arrival if no neighbor is already occupied. This is terminal blocking random sequential adsorption, as studied by Penrose and Sudbury~\cite{PS}.

We consider the variance of the accepted cardinality. Throughout, graphs are finite, simple, and undirected; they need not be connected. Let $X_G$ denote the accepted count, and write $e(G)=|E(G)|$.

\begin{theorem}\label{thm:main}
If $G$ is triangle-free and $n=|V(G)|\ge2$, then
\begin{equation}\label{eq:main}
 \Var(X_G)\le e(G)\left(\frac{n-2}{n}\right)^2.
\end{equation}
Equality holds if and only if $G$ is edgeless or $G\cong K_{1,n-1}$.
\end{theorem}

\begin{corollary}\label{cor:tree}
For every nonempty tree $T$ of order $n$,
\[
 \Var(X_T)\le \frac{(n-1)(n-2)^2}{n^2},
\]
with equality if and only if $T\cong K_{1,n-1}$. The singleton is included as $K_{1,0}$.
\end{corollary}

Gamarnik and Goldberg~\cite[Theorem~6, equation~(9)]{GG} give $\Var(X_G)\le C(r)n$ under a maximum-degree bound $r\ge3$, without a girth restriction. Penrose and Sudbury~\cite[Section~2.7, discussion following Theorem~7]{PS} also note a linear variance upper bound in the uniformly bounded-degree setting, alongside their correlation and normal-approximation results. Krivelevich, M\'esz\'aros, Michaeli, and Shikhelman~\cite[Theorem~1.2]{KMMS} prove concentration under subfactorial path growth and local convergence. Their Theorem~1.4 makes the path mean-minimal among trees, while Section~8 makes the star uniquely mean-maximal. Corollary~\ref{cor:tree} gives the variance analogue of the latter extremal statement; a bound on the mean alone does not control the variance. In terms of average degree $\bar d=2e(G)/n$, \eqref{eq:main} reads
\[
 \Var(X_G/n)\le \frac{\bar d}{2n}\left(1-\frac2n\right)^2.
\]
Our bound imposes no restriction on the maximum degree.

The degree-free baseline $\Var(X_G)\le e(G)$ is already an elementary consequence of the expected deletion stability of Censor-Hillel, Haramaty, and Karnin~\cite{CHHK}, together with first-choice conditioning and a triangle-free edge count; we give this deduction in Proposition~\ref{prop:baseline}. Related dynamic-MIS and sensitivity methods appear in~\cite{CZ,VY}. Thus the point of Theorem~\ref{thm:main} is the sharp finite-order coefficient and equality classification, not the existence of a linear-in-edge-count variance bound. The refinement uses a stronger centered first-choice estimate. We include a proof of the needed deletion-stability statement for self-containment.

Sudbury~\cite[Theorem~1]{Sudbury} compares single-vertex occupation probabilities on regular and branching-random trees. For a uniformly random labelled tree, Panholzer~\cite[Theorem~4]{Panholzer} obtains an exact accepted-count distribution, variance $n/16+O(1)$, and a central limit theorem. These are averaged over the random tree as well as the arrival order; they do not by themselves give a variance bound for every fixed tree.

\section{Deletion stability and first-choice means}

Write $I(G,\pi)$ for the independent set obtained from an order $\pi$, and put
\[
 \mu(G)=\E X_G,\qquad v(G)=\Var(X_G),\qquad
 \mu(\varnothing)=v(\varnothing)=0.
\]
When comparing induced subgraphs, use the restrictions of the same uniform permutation. Let $N_G(x)$ and $N_G[x]$ be the open and closed neighborhoods, and let $d_G(x)=|N_G(x)|$. For a nonempty graph set
\[
 R_u=G-N_G[u],\qquad a_u(G)=1+\mu(R_u),\qquad
 b_u(G)=\mu(G)-a_u(G).
\]
Conditional on $u$ being first, the remaining accepted count has the law of $X_{R_u}$. Therefore, with $n=|V(G)|$,
\begin{equation}\label{eq:mean}
 \mu(G)=\frac1n\sum_u a_u(G),\qquad \sum_u b_u(G)=0.
\end{equation}
Define $c_0=c_1=0$ and $c_n=(n-2)/n$ for $n\ge2$. Then $0\le c_n\le1$, and $c_n$ is strictly increasing for $n\ge2$.

\begin{lemma}[Deletion stability, {\cite[Theorem~1]{CHHK}}]\label{lem:deletion}
For any graph $G$ and fixed vertex $x\in V(G)$,
\begin{equation}\label{eq:stability}
 \E|I(G,\pi)\triangle I(G-x,\pi)|\le1.
\end{equation}
Consequently $|\mu(G)-\mu(G-x)|\le1$, and for subsets $A,B\subseteq V(G)$,
\begin{equation}\label{eq:lipschitz}
 |\mu(G[A])-\mu(G[B])|\le |A\triangle B|.
\end{equation}
\end{lemma}

\begin{proof}
Induct on $n=|V(G)|$, the singleton case being immediate. Fix $x$, write $d=d_G(x)$, and condition on the first vertex $u$. The remaining relative order is uniform.

If $u\in N_G(x)$, both runs accept $u$ and thereafter reduce to the same graph $R_u$, so their discrepancy is zero. If $u\notin N_G[x]$, both runs accept $u$ and the discrepancy equals that for deleting $x$ from $R_u$. Its conditional expectation is at most one by induction.

If $u=x$, the full run gives $\{x\}\cup I(R_x,\pi)$, whereas the deleted run gives $I(G-x,\pi)$. Delete the $d$ neighbors of $x$ successively from $G-x$, in any fixed order. The triangle inequality bounds the symmetric difference by one plus the sum of these $d$ deletion discrepancies. Each intermediate graph is smaller than $G$ and has a uniform restricted order, so the conditional expectation is at most $1+d$. Averaging gives
\[
 \E|I(G,\pi)\triangle I(G-x,\pi)|
 \le\frac{d\cdot0+(n-d-1)\cdot1+(d+1)}n=1.
\]
The mean assertion follows by bounding cardinality difference by symmetric difference. Telescoping deletions from $A$ and $B$ to $A\cap B$ proves \eqref{eq:lipschitz}.
\end{proof}

Deleting the open neighborhood $N_G(x)$ leaves $x$ isolated. Hence, for every graph $G$ and vertex $x$, \eqref{eq:lipschitz} gives the elementary centered estimate
\begin{equation}\label{eq:crude-bias}
 |b_x(G)|=|\mu(G)-\mu(G-N_G(x))|\le d_G(x),
 \qquad \mu(G-N_G(x))=1+\mu(R_x).
\end{equation}
The next lemma sharpens this estimate for triangle-free graphs.

\begin{lemma}[Centered first-choice estimate]\label{lem:centered}
For every triangle-free graph $G$ of order $n\ge2$ and every $x\in V(G)$,
\begin{equation}\label{eq:centered}
 |b_x(G)|\le c_n d_G(x).
\end{equation}
For a singleton, $b_x(G)=0$. Hence, for a uniform vertex $u$,
\begin{equation}\label{eq:between}
 \Var_u(a_u(G))=\frac1n\sum_u b_u(G)^2
 \le\frac{c_n^2}{n}\sum_u d_G(u)^2.
\end{equation}
\end{lemma}

\begin{proof}
The singleton case is immediate; assume $n\ge2$.

Fix $x$, put $d=d_G(x)$, and let $U=V(G)\setminus N_G[x]$. We first establish
\begin{equation}\label{eq:recurrence}
 n b_x(G)=\sum_{u\in N_G(x)}\bigl(\mu(R_u)-\mu(R_x)\bigr)
                  +\sum_{u\in U}b_x(R_u).
\end{equation}
For $u\in U$, deleting the two closed neighborhoods commutes:
\[
 R_u-N_{R_u}[x]=G-(N_G[u]\cup N_G[x])=R_x-N_{R_x}[u].
\]
Thus $\mu(R_u)=1+\mu(R_x-N_{R_x}[u])+b_x(R_u)$. Sum over $u\in U=V(R_x)$ and apply \eqref{eq:mean} to $R_x$ to obtain
\[
 \sum_{u\in U}\mu(R_u)=|U|\mu(R_x)+\sum_{u\in U}b_x(R_u).
\]
When $U=\varnothing$ both sides are zero. Substitution into $n\mu(G)=n+\sum_u\mu(R_u)$ gives \eqref{eq:recurrence}.

If $u$ is adjacent to $x$, triangle-freeness gives $N_G[u]\cap N_G[x]=\{u,x\}$. By \eqref{eq:lipschitz},
\[
 |\mu(R_u)-\mu(R_x)|\le d_G(u)+d-2.
\]
For $u\in U$, apply \eqref{eq:crude-bias} to $R_u$ to get $|b_x(R_u)|\le d_{R_u}(x)$. A count of ordered pairs $(u,w)$, where $w\in N_G(x)$ is removed by deleting $N_G[u]$, gives
\begin{equation}\label{eq:degreecount}
 \sum_{u\in U}d_{R_u}(x)
 =(n-d-1)d-\sum_{w\in N_G(x)}(d_G(w)-1).
\end{equation}
Indeed, for each such $w$, every vertex of $N_G(w)\setminus\{x\}$ lies in $U$, because there are no triangles. The count allows different common neighbors, so it does not exclude four-cycles.

Taking absolute values in \eqref{eq:recurrence} now gives
\begin{align*}
 n|b_x(G)|
 &\le \sum_{u\in N_G(x)}(d_G(u)+d-2)
 +(n-d-1)d-\sum_{w\in N_G(x)}(d_G(w)-1)\\
 &=d(n-2).
\end{align*}
This proves \eqref{eq:centered}, also for $n=2$ and for isolated $x$. Equation \eqref{eq:between} follows from \eqref{eq:mean}.
\end{proof}

\section{Variance and equality}

\begin{lemma}[Edge loss]\label{lem:edges}
For a nonempty triangle-free graph $G$ of order $n$,
\begin{equation}\label{eq:edges}
 \frac1n\sum_u e(R_u)+\frac1n\sum_u d_G(u)^2=e(G).
\end{equation}
\end{lemma}

\begin{proof}
Every edge removed by deleting $N_G[u]$ is incident to a neighbor of $u$. No edge joins two such neighbors, so
\[
 e(G)-e(R_u)=\sum_{w\in N_G(u)}d_G(w).
\]
Summing over $u$ proves \eqref{eq:edges}.
\end{proof}

\begin{proposition}[Elementary edge-count bound]\label{prop:baseline}
For every finite triangle-free graph $G$, $v(G)\le e(G)$.
\end{proposition}

\begin{proof}
Induct on the order, with $v(\varnothing)=e(\varnothing)=0$. For a nonempty graph, condition on its first vertex. Using induction on the smaller residual graphs, \eqref{eq:crude-bias}, and Lemma~\ref{lem:edges},
\[
 v(G)=\frac1n\sum_u v(R_u)+\frac1n\sum_u b_u(G)^2
 \le\frac1n\sum_u e(R_u)+\frac1n\sum_u d_G(u)^2=e(G).
\]
This argument uses the known deletion estimate, not Lemma~\ref{lem:centered} or Theorem~\ref{thm:main}.
\end{proof}

\begin{proof}[Proof of Theorem~\ref{thm:main}]
We first prove the bound by induction on the order over all triangle-free graphs. Counts are deterministic for orders zero, one, and two. For $n\ge3$, the law of total variance, conditional on the first vertex, gives
\begin{equation}\label{eq:variance}
 v(G)=\frac1n\sum_u v(R_u)+\frac1n\sum_u b_u(G)^2.
\end{equation}
The first term is the expected conditional variance; the second is the variance of the conditional mean. With $s_u=|V(R_u)|<n$, induction and monotonicity give $v(R_u)\le c_{s_u}^2e(R_u)\le c_n^2e(R_u)$, including empty residual graphs. Applying \eqref{eq:between} and \eqref{eq:edges} to \eqref{eq:variance} yields
\[
 v(G)\le\frac{c_n^2}{n}\left(\sum_u e(R_u)+\sum_u d_G(u)^2\right)
       =c_n^2 e(G).
\]

For equality, edgeless graphs attain the zero bound. If $e(G)>0$ and $n=2$, then $G=K_{1,1}$. Suppose $n\ge3$. If some $R_u$ has an edge, then $2\le s_u<n$ and $c_{s_u}<c_n$, making the corresponding induction estimate strict. Thus equality requires every $R_u$ to be edgeless.

Choose an edge $xy$. Every other vertex must be adjacent to at least one of $x,y$, since otherwise $xy$ survives in its residual graph; triangle-freeness prevents adjacency to both. Therefore $A=N_G(y)$ and $B=N_G(x)$ partition $V(G)$ into independent sets, with $x\in A$ and $y\in B$. If $z\in A\setminus\{x\}$ and $w\in B$ were nonadjacent, then the edge $xw$ would survive in $R_z$, a contradiction. Hence $G=K_{a,b}$ for positive integers $a=|A|$, $b=|B|$.

The first vertex of $K_{a,b}$ determines the accepted part. Its count is $a$ with probability $a/(a+b)$ and $b$ with probability $b/(a+b)$, so
\[
 v(K_{a,b})=\frac{ab(a-b)^2}{(a+b)^2},\qquad e(K_{a,b})=ab.
\]
As $a+b=n$, equality in \eqref{eq:main} is equivalent to $|a-b|=n-2$, or $\min(a,b)=1$. These are exactly the connected stars, and the same computation proves that they attain equality.
\end{proof}

\begin{proof}[Proof of Corollary~\ref{cor:tree}]
A tree of order $n\ge2$ is triangle-free and has $n-1$ edges, so apply Theorem~\ref{thm:main}. For $n=1$, the count is identically one and the bound is zero.
\end{proof}

\section{Scope of the bound}

Theorem~\ref{thm:main} is sharp uniformly and gives the exact maximum over trees of each order. It does not give an attained maximum for every prescribed pair $(n,e)$. Nor does it assert a central limit theorem.

Relative to Proposition~\ref{prop:baseline}, the finite-order improvement is $4e(G)(n-1)/n^2$. For a tree it is
\[
 (n-1)-\frac{(n-1)(n-2)^2}{n^2}=4\left(1-\frac1n\right)^2,
\]
which tends to four. The coefficient one in the baseline cannot be decreased uniformly in $n$: the explicit star law gives $v(K_{1,n-1})/e(K_{1,n-1})=(1-2/n)^2\to1$. The refinement supplies the exact finite-order tree maximum and its equality cases.

Triangle-freeness is used in both \eqref{eq:degreecount} and \eqref{eq:edges}. If $\tau(G)$ is the number of triangles in an arbitrary graph, the latter identity becomes
\[
 \frac1n\sum_u e(R_u)+\frac1n\sum_u d_G(u)^2
 =e(G)+\frac{3\tau(G)}n.
\]
Indeed, the removed-edge count is $\sum_{w\in N_G(u)}d_G(w)-e(G[N_G(u)])$, and $\sum_u e(G[N_G(u)])=3\tau(G)$. The argument above does not settle whether \eqref{eq:main} continues to hold when triangles are present.

\end{document}